\documentclass[final]{siamltex}
\usepackage{latexsym,bm,amssymb,amsfonts,amsbsy,amsmath,graphics,graphicx,color}
\usepackage{url}
\def\dd{\,\mathrm{d}}
\newtheorem{corol}{Corollary}
\newtheorem{remark}{Remark}

\def \CC {{\mathbb{C}}}
\def \RR {{\mathbb{R}}}

\def \P {\hat{P}}

\def \pmatrix{ \left( \begin{array} }
\def \endpmatrix{ \end{array} \right) }
\def\aa{\alpha}
\def\bb{\beta}
\def\i{\mathrm{i}}

\begin{document}

\title{A unifying framework for the derivation and analysis of effective
classes of one-step methods for ODEs\,\thanks{Work developed
within the project ``Numerical methods and software for
differential equations''.}}

\author{Luigi Brugnano\thanks{Dipartimento di Matematica
``U.\,Dini'', Universit\`a di Firenze, Italy ({\tt
luigi.brugnano@unifi.it}).} \and Felice
Iavernaro\thanks{Dipartimento di Matematica, Universit\`a di Bari,
Italy ({\tt felix@dm.uniba.it}).} \and Donato
Trigiante\thanks{Dipartimento di Energetica ``S.\,Stecco'', Universit\`a di
Firenze, Italy ({\tt trigiant@unifi.it}).}}

\maketitle

\begin{abstract} In this paper, we provide a simple
framework to derive and analyse several classes of effective
one-step methods. The framework consists in the discretization of
a local Fourier expansion of the continuous problem. Different
choices of the basis lead to different classes of methods, even
though we shall here consider only the case of an orthonormal
polynomial basis, from which a large subclass of Runge-Kutta methods
can be derived. The obtained results are then applied to prove, in a
simplified way, the order and stability properties of Hamiltonian BVMs
(HBVMs), a recently introduced class of energy preserving methods
for canonical Hamiltonian systems (see \cite{BIT0} and references
therein). A few numerical tests with such methods are also
included, in order to confirm the effectiveness of the methods.
\end{abstract}



\section{Introduction}
\begin{flushright}\vspace{-1em}
  \textsf{\em Though I have not always been able to make simple}\\
  \textsf{\em a difficult thing, I never made difficult a simple one.}\\
  \textsf{F.\,G.\,Tricomi}\hspace{5.75cm}~\\
\end{flushright}

One-step methods are widely used in the numerical solution of
initial value problems for ordinary differential equations which,
without loss of generality, we shall assume to be in the form:
\begin{equation}\label{ivp}
y'(t) = f(y(t)), \qquad t\in[t_0,t_0+T],\qquad y(t_0)=y_0\in\RR^m.
\end{equation}
In particular, we consider a very general class of effective
one-step methods that can be led back to a local Fourier expansion
of the continuous problem over the interval $[t_0,t_0+h]$, where $h$
is the considered stepsize. In general, different choices of the
basis result in different classes of methods, for which, however,
the analysis turns out to be remarkably simple. Though the arguments
can be extended to a general choice of the basis, we consider here
only the case of a polynomial basis, from which one obtains a large
subclass of Runge-Kutta methods. Usually, the order properties of
such methods are studied through the classical theory of Butcher on
rooted trees (see, e.g., \cite[Th.\,2.13 on p.\,153]{HNW}), almost
always resorting to the so called {\em simplifying assumptions}
(see, e.g., \cite[Th.\,7.4 on p.\,208]{HNW}). Nonetheless, such
analysis turns out to be greatly simplified for the methods derived
in the new framework, which is introduced in
Section~\ref{fouriersec}. Similar arguments apply to the linear
stability analysis of the methods, here easily discussed through the
Lyapunov method. Then, we apply the same procedure to the case where
(\ref{ivp}) is a canonical Hamiltonian problem, i.e., a problem in
the form
\begin{equation}\label{ham}
  \frac{dy}{dt}= J \nabla H(y), \qquad
  J=\pmatrix{cc}0&I_m\\-I_m&0\endpmatrix, \qquad y(t_0)=y_0\in\RR^{2m},
\end{equation}
where $H(y)$ is a smooth scalar function, thus obtaining, in
Section~\ref{mfe}, an alternative derivation of the recently
introduced class of energy preserving methods called Hamiltonian
BVMs (HBVMs, see \cite{BIT0,BIT1,BIT2} and references therein). A
few numerical examples concerning such methods are then provided
in Section~\ref{numtest}, in order to make evident their
potentialities. Some concluding remarks are then given in
Section~\ref{final}.

\section{Local Fourier expansion of ODEs}\label{fouriersec}
Let us consider problem (\ref{ivp}) restricted to the interval
$[t_0,t_0+h]$:
\begin{equation}\label{localivp}
y' = f(y), \qquad t\in[t_0,t_0+h], \qquad y(t_0)=y_0.
\end{equation}
In order to make the arguments as simple as possible, we shall
hereafter assume $f$ to be analytical. Then, let us fix an
orthonormal basis $\{\P_j\}_{j=0}^\infty$ over the interval $[0,1]$,
even though different bases and/or reference intervals could be in
principle considered. In particular, hereafter we shall consider a
polynomial basis: i.e., the shifted Legendre polynomials over the
interval $[0,1]$, scaled in order to be orthonormal. Consequently,
$$\int_0^1 \P_i(x)\P_j(x) \dd x = \delta_{ij}, \qquad \deg\P_j = j,
\qquad \forall i,j\ge0,$$ where $\delta_{ij}$ is the Kronecker
symbol. We can then rewrite (\ref{localivp}) by expanding the
right-hand side:
\begin{equation}
y'(t_0+ch) = \sum_{j=0}^\infty \P_j(c) \gamma_j(y), \quad c\in[0,1];
\label{localivpexp} \qquad \gamma_j(y) = \int_0^1 \P_j(\tau)
f(y(t_0+\tau h)) \dd\tau.
\end{equation}
The basic idea (first sketched in \cite{BIT5}) is now that of
truncating the series iafter $r$ terms, which turns
(\ref{localivpexp}) into
\begin{equation}\label{localivpexpr}
\omega'(t_0+ch) = \sum_{j=0}^{r-1} \P_j(c) \gamma_j(\omega), \quad
c\in[0,1]; \qquad \gamma_j(\omega) = \int_0^1 \P_j(\tau)
f(w(t_0+\tau h)) \dd\tau.
\end{equation}
By imposing the initial condition, one then obtains
\begin{equation}\label{solr}
\omega(t_0+ch) = y_0+h\sum_{j=0}^{r-1} \gamma_j(\omega) \int_0^c \P_j(x)
 \dd x, \qquad c\in[0,1].
\end{equation}
Obviously, $\omega$ is a polynomial of degree at most $r$. The
following question then naturally arises: ~{\em ``how close are
$y(t_0+h)$ and $\omega(t_0+h)$?''}~ The answer is readily obtained, by using the following preliminary result.

\begin{lemma}\label{lem1}
Let $g:[0,h]\rightarrow \RR^m$ be a suitably regular function.
Then $\int_0^1 \P_j(\tau)g(\tau h) \dd \tau = O(h^j).$
\end{lemma}

\begin{proof} Assume, for sake of simplicity, $$g(\tau h) =
\sum_{n=0}^\infty \frac{g^{(n)}(0)}{n!} (\tau h)^n$$ to be the
Taylor expansion of $g$. Then, for all $j\ge0$,
$$\int_0^1 \P_j(\tau) g(\tau h) \dd\tau = \sum_{n=0}^\infty
\frac{g^{(n)}(0)}{n!} h^n \int_0^1 \P_j(\tau) \tau^n\dd\tau =
O(h^j),$$ since $\P_j$ is orthogonal to polynomials of degree
$n<j$.~\end{proof}

As a consequence, one has that (see (\ref{localivpexpr})) $\gamma_j(\omega)=O(h^j)$. Moreover, for any given $\tilde t\in[t_0,t_0+h]$, we denote by
$y(s,\tilde t,\tilde y)$ the solution of
\eqref{localivp}-\eqref{localivpexp} at time $s$ and with initial
condition $y(\tilde t)=\tilde y$. Similarly, we denote by\vspace{-.75em}
\begin{equation}\label{lem2}\Phi(s,\tilde t,\tilde
y)=\frac{\partial}{\partial \tilde y}y(s,\tilde t,\tilde
y),
\end{equation}
also recalling the following standard result from the theory of
ODEs:
\begin{equation}\label{lem3}
\frac{\partial}{\partial \tilde t} y(s,\tilde t, \tilde y) =
-\Phi(s,\tilde t, \tilde y)f(\tilde y).
\end{equation}

We can now state the following result, for which we provide a more direct proof,
with respect to that given in \cite{BIT5}. Such proof is essentially
based on that of \cite[Theorem\,6.5.1 on
pp.\,165-166]{LakTri}.

\begin{theorem}\label{fourierord}
Let $y(t_0+ch)$ and $\omega(t_0+ch)$, $c\in[0,1]$, be the
solutions of (\ref{localivpexp}) and (\ref{localivpexpr}),
respectively. Then, ~$y(t_0+h)-\omega(t_0+h) =
O(h^{2r+1}).$
\end{theorem}

\begin{proof} By virtue of Lemma~\ref{lem1} and
\eqref{lem2}-\eqref{lem3}, one has:
\begin{eqnarray*}\lefteqn{y(t_0+h)-\omega(t_0+h) ~=~ y(t_0+h,
t_0, y_0) - y(t_0+h, t_0+h, \omega(t_0+h))}\\
&=& \int_{t_0}^{t_0+h} \frac{d}{d\tau} y(t_0+h, \tau,
\omega(\tau))\dd\tau
~=~\int_{t_0}^{t_0+h} \left(\frac{\partial}{\partial
\tau}y(t_0+h,\tau,\omega(\tau))+\frac{\partial}{\partial
\omega}y(t_0+h,\tau,\omega(\tau))\omega'(\tau)\right) \dd \tau
\\
&=& h\int_0^1 \Phi(t_0+h,t_0+ch,\omega(t_0+ch))\left(
-f(\omega(t_0+ch))+\omega'(t_0+ch)\right)\dd c\\
&=& -h\int_0^1 \Phi(t_0+h,t_0+ch,\omega(t_0+ch))\left( \sum_{j=
r}^\infty \gamma_j(\omega)\P_j(c)\right)\dd c\\
&=& -h \sum_{j=r}^\infty \left( \int_0^1 \P_j(\tau) \Phi(t_0+h,t_0+c
h,\omega(t_0+c h)) \dd c \right) \gamma_j(\omega) ~=~
h\sum_{j=r}^\infty O(h^{j})\,O(h^{j}) ~=~
O(h^{2r+1}).\end{eqnarray*}
\end{proof}

The previous result reveals the
extent to which the polynomial $\omega(t)$, solution of
\eqref{localivpexpr}, approximates the solution $y(t)$ of the
original problem \eqref{localivp} on the time interval
$[t_0,t_0+h]$. Obviously, the value $\omega(t_0+h)$ may serve as the
initial condition for a new IVP in the form \eqref{localivpexpr}
approximating $y(t)$ on the time interval $[t_0+h,t_0+2h]$. In
general, setting $t_i=t_0+ih$, $i=0,1,\dots$, and assuming that an
approximation $\omega(t)$ is available on the interval
$[t_{i-2},t_{i-1}]$, one can extend the approximation to the
interval $[t_{i-1},t_i]$ by solving the IVP
\begin{equation}
\label{globalivpexpr} \omega'(t_{i-1}+ch) = \sum_{j=0}^{r-1} \P_j(c)
\int_0^1 \P_j(\tau) f(\omega(t_{i-1}+ch))\dd\tau,\quad c\in[0,1],
\end{equation}
the initial value $\omega(t_{i-1})$ having been computed at the
preceding step.  The approximation to $y(t)$ is thus extended on an
arbitrary interval $[t_0,t_0+Nh]$,  and the function $\omega(t)$ is
a continuous piecewise polynomial. As a direct consequence of
Theorem \ref{fourierord}, we obtain the following result.

\begin{corol} Let $T=Nh$, where $h>0$ and $N$ is an integer.
Then, the approximation to the solution of problem (\ref{ivp}) by
means of \eqref{globalivpexpr}  at the grid-points $t_i =
t_{i-1}+h$,  $i=1,\dots,N$, with $\omega(t_0)=y_0$, is $O(h^{2r})$
 accurate.
\end{corol}

We now want to compare the asymptotic behavior of $\omega(t)$ and
$y(t)$ on the infinite length interval $[t_0,+\infty)$ in the case
where $f$ is linear or defines a canonical Hamiltonian problem. To
this end we introduce the the infinite sequence $\{\omega_i\} \equiv
\{\omega(t_i)\}$.

\begin{remark}
Though in general, the sequence $\{\omega_i\}$ cannot be formally
regarded as the outcome of a numerical method,  under special
situations, this can be the case. For example, when $f$ is a
polynomial, the integrals in \eqref{localivpexpr} may be explicitly
determined and the IVP in \eqref{localivpexpr} is evidently
equivalent to a nonlinear system having as unknowns the coefficients
of the polynomial $\omega$ expanded along a given basis (for
example, the polynomyal $\omega$ may be computed by means of the
method of undetermined coefficients). This issue, as well as details
about how to manage the integrals in the event that the integrands
do not admit an analytical primitive function in closed form,  will
be thoroughly faced in Section~\ref{mfe}.
\end{remark}

\subsection{Linear stability analysis}\label{stab}

For the linear stability analysis, problem (\ref{localivp}) becomes
the celebrated test equation
\begin{equation}\label{test}
y' = \lambda y, \qquad \Re(\lambda)\le 0.
\end{equation}
By setting
$$\lambda = \aa+\i\bb, \qquad y = x_1 +\i
x_2, \qquad x=(x_1,x_2)^T, \qquad A = \pmatrix{rr} \aa & -\bb\\
\bb &\aa\endpmatrix,
$$
with $\i$ the imaginary unit, problem (\ref{test}) can be rewritten
as
\begin{equation}\label{testx}
x' = A x, \qquad t\in[t_0,t_0+h], \qquad x(t_0) ~\mbox{given}.
\end{equation}
Consequently, the corresponding truncated problem
(\ref{localivpexpr}) becomes
\begin{equation}\label{testr}\omega'(t_0+c h) = A\sum_{j=0}^{r-1}
\P_j(c)\int_0^1 \P_j(\tau) \nabla V(\omega(t_0+\tau h)) \dd\tau,
\qquad c\in[0,1],
\end{equation}
where
\begin{equation}\label{liap} V(x) = \frac{1}2
x^Tx\end{equation} is a Lyapunov function for (\ref{testx}). From
(\ref{testr})-(\ref{liap}) one readily obtains
\begin{eqnarray*}\Delta V(\omega(t_0))
&=&V(\omega(t_0+h))-V(\omega(t_0))
~=~ h\int_0^1 \nabla V(\omega(t_0+\tau h))^T\omega'(t_0+\tau h)\dd\tau\\
&=& h\sum_{j=0}^{r-1} \left[ \int_0^1 \P_j(\tau) \nabla
V(\omega(t_0+\tau h)) \dd\tau\right]^T A \left[ \int_0^1 \P_j(\tau)
\nabla V(\omega(t_0+\tau h)) \dd\tau\right]
\\&=& \aa h\sum_{j=0}^{r-1} \left\| \int_0^1 \P_j(\tau) \omega(t_0+\tau h)
\dd\tau\right\|_2^2.
\end{eqnarray*}
Last equality follows by taking the symmetric part of $A$. We
observe that
$$\omega\ne0 \quad\Longrightarrow\quad \sum_{j=0}^{r-1}
\left\| \int_0^1 \P_j(\tau) \omega(t_0+\tau h) \dd\tau\right\|_2^2>
0,
$$
since, conversely, this would imply $\omega(t_0+c h) =
\rho\cdot \P_r(c)$ for a suitable $\rho\ne0$ and, therefore (from
(\ref{testr})), $\P_r'\equiv 0$~ which is clearly not true. Thus for
a generic $y_0\not =0$,
$$
\Delta V(\omega(t_0))<0 \Longleftrightarrow \Re(\lambda)<0 \quad
\mbox{and} \quad \Delta V(\omega(t_0))=0 \Longleftrightarrow
\Re(\lambda)=0.
$$
Again, the above computation can be extended to any interval
$[t_{i-1},t_i]$ and, from the discrete version of the Lyapunov
theorem  (see, e.g., \cite[Th.\,4.8.3 on p.\,108]{LakTri}) , we have
that the sequence $\omega_i$ tends to zero if and only if
$\Re(\lambda)<0$, while it remains bounded whenever
$\Re(\lambda)=0$, whatever is the stepsize $h>0$ used. The following
result is thus proved.
\begin{theorem}\label{teostab} The continuous solution $y(t)$ of
\eqref{test} and its discrete approximation $\omega_i$ have the same
stability properies, for any choice of the stepsize $h>0$.
\end{theorem}

\subsection{The Hamiltonian case}\label{inftyhbvms}

In the case where problem (\ref{ivp}) is Hamiltonian, i.e.,
(\ref{ham}), the approximation provided by the polynomial $\omega$
in  (\ref{localivpexpr})-(\ref{solr}) inherits a very important
property of the continuous problem, i.e., energy conservation.
Indeed, it is very well known that for the continuous solution one
has, by virtue of (\ref{ham}),
$$\frac{\dd}{\dd t} H(y(t)) = \nabla H(y(t))^Ty'(t) =
\nabla H(y(t))^TJ \nabla H(y(t)) = 0,
$$
due to the fact that matrix
$J$ is skew-symmetric. Consequently, $H(y(t))=H(y_0)$ for all $t$.
For the truncated Fourier problem, the following result holds true.

\begin{theorem}\quad
$H(\omega(t_0+h)) = H(\omega(t_0)) \equiv H(y_0).$
\end{theorem}

\begin{proof} From (\ref{localivpexpr}), considering that $f(\omega) =
J\nabla H(\omega)$ and $J^TJ=I$, one obtains:
\begin{eqnarray*} \lefteqn{H(\omega(t_0+h))-H(y_0) =}\\ &=& h\int_0^1 \nabla
H(\omega(t_0+\tau h))^T \omega'(t_0+\tau h)\dd\tau
~=~ h\int_0^1 \nabla H(\omega(t_0+\tau h))^T \sum_{j=0}^{r-1}
\P_j(\tau) \gamma_j(\omega)\dd\tau\\
&=& h\sum_{j=0}^{r-1}\left(\int_0^1 \nabla H(\omega(t_0+\tau h))
\P_j(\tau)\dd\tau\right)^T \gamma_j(\omega) ~=~ h\sum_{j=0}^{r-1}
\gamma_j(\omega)^TJ\gamma_j(\omega) ~=~0,
\end{eqnarray*} since $J$ is
skew-symmetric.~\end{proof}

\section{Discretization}\label{mfe}

Clearly, the integrals in (\ref{localivpexpr}), if not directly
computable, need to be numerically approximated. This can be done by
introducing a quadrature formula based at $k\ge r$ abscissae
$\{c_i\}$, thus obtaining an approximation to (\ref{localivpexpr}):
\begin{equation}\label{localivpexprk}
u'(t_0+c h)=\sum_{j=0}^{r-1} \P_j(c)
 \, \sum_{\ell=1}^k b_\ell \P_j(c_\ell) f(u(t_0+c_\ell h)), \qquad
c\in[0,1].
\end{equation}
where the $\{b_\ell\}$ are the quadrature weights, and $u$ is the
resulting polynomial, of degree at most $r$, approximating $\omega$.
It can be obtained by solving a discrete problem in the form:
\begin{equation}\label{hbvmkr}
u'(t_0+c_ih)=\sum_{j=0}^{r-1} \P_j(c_i)
 \, \sum_{\ell=1}^k b_\ell \P_j(c_\ell) f(u(t_0+c_\ell h)), \qquad
i=1,\dots,k.
\end{equation}
 Let $q$ be the order of the formula, i.e., let it be exact
for polynomials of degree less than $q$ (we observe that $q\ge k\ge
r$). Clearly, since we assume $f$ to be analytical, choosing $k$
large enough, along with a suitable choice of the nodes $\{c_i\}$,
allows to approximate the given integral to any degree of accuracy,
even though, when using finite precision arithmetic, it suffices to
approximate it to machine precision. We observe that, since the
quadrature is exact for polynomials of degree $q-1$, then its
remainder depends on the $q$-th derivative of the integrand with
respect to $\tau$. Consequently, considering that
$\P_j^{(i)}(c)\equiv 0$, for $i>j$, one has
\begin{equation}\label{errquad} \Delta_j(h) ~\equiv~
\int_0^1 \P_j(\tau) f(u(t_0+\tau h)) \mathrm{d} \tau-
\sum_{\ell=1}^k b_\ell \P_j(c_\ell) f(u(t_0+c_\ell h)) ~=~
O(h^{q-j}),
\end{equation}
$j=0,\dots,r-1$.~ Thus, (\ref{localivpexprk}) is equivalent to the
ODE,
\begin{equation}\label{omegacprimemod}
 u'(t_0+ch)= \sum_{j=0}^{r-1} \P_j(c)\left(\gamma_j(u)
-\Delta_j(h)\right),\quad c\in[0,1], \quad  \gamma_j(u)=\int_0^1 \P_j(\tau) f(u(t_0+\tau h)) \mathrm{d}
\tau,
\end{equation}
with $u(t_0)=y_0$, in place of (\ref{localivpexpr}). The
following result then holds true.

\begin{theorem} In the above hypotheses:\,
$y(t_0+h) - u(t_0+h) = O(h^{p+1})$,\,  with\, $p=\min(q,2r)$.
\end{theorem}

\begin{proof} The proof is quite similar to that of
Theorem~\ref{fourierord}: by virtue of Lemma~\ref{lem1} and
(\ref{errquad})-(\ref{omegacprimemod}), one obtains
\begin{eqnarray*}\lefteqn{y(t_0+h)-u(t_0+h) ~=~ y(t_0+h,
t_0, y_0) - y(t_0+h, t_0+h, u(t_0+h))}\\
&=& \int_{t_0}^{t_0+h} \frac{d}{d\tau} y(t_0+h, \tau,
u(\tau))\dd\tau
~=~\int_{t_0}^{t_0+h} \left(\frac{\partial}{\partial
\tau}y(t_0+h,\tau,u(\tau))+\frac{\partial}{\partial
u}y(t_0+h,\tau,u(\tau))u'(\tau)\right) \dd \tau
\\
&=& h\int_0^1 \Phi(t_0+h,t_0+ch,u(t_0+ch))\left(
-f(u(t_0+ch))+u'(t_0+ch)\right)\dd c\\
&=& h\int_0^1 \Phi(t_0+h,t_0+ch,u(t_0+ch))\left( \sum_{j=0}^{r-1} \P_j(c)\Delta_j(h)-\sum_{j=
r}^\infty \gamma_j(u)\P_j(c)\right)\dd c\\
&=& h \sum_{j=0}^{r-1} \left( \int_0^1 \P_j(\tau)
\Phi(t_0+h,t_0+c h,u(t_0+c h)) \dd c
\right) \Delta_j(u)\\&&-\,h \sum_{j=r}^\infty \left( \int_0^1 \P_j(\tau)
\Phi(t_0+h,t_0+c h,u(t_0+c h)) \dd c
\right) \gamma_j(u)\\
&=& h\sum_{j=0}^{r-1} O(h^{j})\,O(h^{q-j}) \,-\,h\sum_{j=r}^\infty
O(h^{j})\,O(h^{j}) ~=~ O(h^{q+1})+O(h^{2r+1}).
\end{eqnarray*}\end{proof}

As an immediate consequence, one has the following result.

\begin{corol}\label{ord} Let $q$ be the order of the quadrature formula defined by the
abscissae $\{c_i\}$. Then, the order of the method (\ref{hbvmkr})
for approximating (\ref{ivp}), with $y_1=u(t_0+h)$, is
$p=\min(q,2r)$.
\end{corol}

Concerning the linear stability analysis, by considering that a
quadrature formula of order $q\ge 2r$ is exact when the integrand
is a polynomial of degree at most $2r-1$, the following result
immediately derives from Theorem~\ref{teostab}.

\begin{corol}\label{stabcor} Let $q$ be the order of the quadrature formula defined by the
abscissae $\{c_i\}$. If $q\ge 2r$, then the method (\ref{hbvmkr})
is perfectly $A$-stable.\footnote{I.e., its absolute stability region coincides with the left-half complex plane, $\CC^-$, \cite{BTbook}.}
\end{corol}

In the case $r=1$, the above results apply to the methods in \cite{IP} (see
also \cite{IT3}).

\subsection{Runge-Kutta formulation}\label{rksec}
By setting, as usual, ~$u_i=u(t_0+c_ih)$, ~$f_i=f(u_i)$,
~$i=1,\dots,k$,~ (\ref{hbvmkr}) can be rewritten as
\begin{equation}\label{rk}
u_i=y_0+h\sum_{j=0}^{r-1} \int_0^{c_i}\P_j(\tau)\dd\tau
 \, \sum_{\ell=1}^k b_\ell \P_j(c_\ell) f_\ell, \qquad
i=1,\dots,k.
\end{equation} Moreover, since $q\ge r\ge \deg u$, one
has $y_1=u(t_0+h) \equiv y_0+h\sum_{\ell=1}^k b_\ell f_\ell$.
Consequently, the methods which Corollary~\ref{ord} refers to are
the subclass of Runge-Kutta methods with the following tableau:
\begin{equation}\label{class}
\begin{array}{c|c}
\begin{array}{c} c_1\\ \vdots \\ c_k\end{array} & A=(a_{ij}) \equiv
\left( b_j\sum_{\ell=0}^{r-1} \P_\ell(c_j)\int_0^{c_i}\P_\ell(\tau)\dd\tau \right)\\
\hline & \begin{array}{ccc} b_1
&\dots&b_k\end{array}\end{array}.
\end{equation} In particular, in
\cite{BIT3_1} it has been proved that when the nodes $\{c_i\}$
coincide with the $k$ Gauss points in the interval $[0,1]$, then $$A
= {\cal A} {\cal P} {\cal P}^T\Omega,$$ where ${\cal
A}\in\RR^{k\times k}$ is the matrix in the Butcher tableau of the
$k$-stages Gauss method, ${\cal P}=(\P_{j-1}(c_i))\in\RR^{k\times
r}$, and $\Omega=\diag(b_1,\dots,b_k)$. In such a way, when $k=r$,
one obtains the classical $r$-stages Gauss collocation method.
Consequently, (\ref{class}) can be regarded as a generalization of
the classical Runge-Kutta collocation methods, (\ref{hbvmkr}) being
interpreted as {\em extended collocation conditions}.

\subsection{Hamiltonian Boundary Value Methods (HBVMs)} \label{hbvmsec}

When considering a canonical Hamiltonian problem (\ref{ham}),  the
discretization of the integrals appearing in (\ref{localivpexpr})
by means of a Gaussian formula at $k$ nodes results in the
HBVM$(k,r)$ methods introduced in \cite{BIT2}.\footnote{A
different discretization, based at $k+1$ Lobatto abscissae, was
previously considered in \cite{BIT1}.} For such methods we derive,
in a novel way with respect to \cite{BIT0,BIT1,BIT2}, the
following result.

\begin{corol}
For all $k\ge r$, HBVM$(k,r)$ is perfectly $A$-stable and has
order $2r$. The method is energy conserving for all polynomial
Hamiltonians of degree not larger than $2k/r$.
\end{corol}

\begin{proof} The result on the order and linear stability easily follow
from Corollaries~\ref{ord} and \ref{stabcor}, respectively.
Concerning energy conservation, one has
\begin{eqnarray*}
\lefteqn{H(u(t_0+h)) -H(y_0) ~=~
 h\int_0^1 \nabla H(u(t_0+\tau h))^T u'(t_0+\tau h) \dd\tau}\\
&=& h\int_0^1 \nabla H(u(t_0+\tau h))^T \sum_{j=0}^{r-1}
\P_j(\tau)
\sum_{\ell=1}^k b_\ell \P_j(c_\ell) J\nabla H(u(t_0+c_\ell h)) \dd\tau\\
&=& h\sum_{j=0}^{r-1} \left[\int_0^1 \P_j(\tau) \nabla H(u(t_0+\tau
h))\dd\tau \right]^T J \left[\sum_{\ell=1}^k b_\ell \P_j(c_\ell)
\nabla H(u(t_0+c_\ell h))\right] ~=~ 0,
\end{eqnarray*}
provided that
\begin{equation}\label{exint}\int_0^1 \P_j(\tau)
\nabla H(u(t_0+\tau h))\dd\tau = \sum_{\ell=1}^k b_\ell \P_j(c_\ell)
\nabla H(u(t_0+c_\ell h)).
\end{equation}
In the case where $H$ is a
polynomial of degree $\nu$, this is true provided that the integrand
is a polynomial of degree at most $2k-1$. Consequently,~ $\nu r-1
\le 2k-1$,~ i.e.,~ $\nu\le 2k/r$.
\end{proof}

 In the case of general Hamiltonian problems, if we consider the limit
as $k\rightarrow\infty$ we recover formulae (\ref{localivpexpr}),
which have been called {\em HBVM$(\infty,r)$} (or, more in general,
{\em $\infty$-HBVMs}) \cite{BIT0,BIT2}: by the way, (\ref{solr}) is
nothing but the {\em Master Functional Equation} in
\cite{BIT0,BIT2}. In the particular case $r=1$, we derive the {\em
averaged vector field} in \cite{QMcL} (for a related approach see
\cite{H}).

\begin{remark} We observe that, in the case of polynomial Hamiltonian
systems, if (\ref{exint}) holds true for $k=k^*$, then
$${\rm
HBVM}(k,r) \equiv {\rm HBVM}(k^*,r) \equiv {\rm HBVM}(\infty,r),
\qquad \forall k\ge k^*.$$
That is, (\ref{localivpexprk}) coincides
with (\ref{localivpexpr}), for all $k\ge k^*$. In the non polynomial
case,  the previous conclusions continue ``practically'' to hold, provided that the
integrals are approximated within machine precision.\end{remark}

\begin{remark} As is easily deduced from the arguments in
Section~\ref{rksec}, the HBVM$(r,r)$ method is nothing but the
$r$-stages Gauss method of order $2r$ (see also \cite{BIT2}).\end{remark}

\section{Numerical Tests}\label{numtest} We here provide a few numerical tests,
showing the effectiveness of HBVMs, namely of the methods obtained
in the new framework, when the problem (\ref{ivp}) is in the form
(\ref{ham}). In particular, we consider the Kepler problem, whose
Hamiltonian is (see, e.g., \cite[p.\,9]{HLW}):\\
\centerline{$H\left([q_1,\,q_2,\,p_1,\,p_2]^T\right) =
\tfrac{1}{2}\left(
p_1^2+p_2^2\right)-\left(q_1^2+q_2^2\right)^{-\frac{1}2}.$}
 When started at\\
\centerline{$\left(~ 1-e,\quad 0,\quad 0,\quad \sqrt{(1+e)/(1-e)}
~\right)^T, \qquad e\in[0,1),$} it has an elliptic periodic orbit
of period $2\pi$ and eccentricity $e$.  When $e$ is close to 0,
the problem is efficiently solved by using a constant stepsize.
However, it becomes more and more difficult as $e \rightarrow 1$,
so that a variable-step integration would be more appropriate in
this case. We now compare the following 6-th order methods for
solving such problem over a 1000 periods interval:
\begin{itemize}
 \item HBVM(3,3), i.e., the GAUSS6 method, which is a symmetric and symplectic
method;

 \item HBVM(4,3), which is symmetric \cite{BIT1} but not symplectic nor energy
preserving, since the Gauss quadrature formula of order 8 is not enough
accurate, for this problem;

 \item HBVM(15,3), which is {\em practically} energy preserving, since the
Gauss formula of order 30 is accurate to machine precision, for this
problem.
\end{itemize}
In the two plots in Figure~\ref{cost_hy} we see the obtained
results when $e=0.6$ and a constant stepsize is used: as one can
see, the Hamiltonian is approximately conserved for the GAUSS6 and
HBVM(4,3) methods, and (practically) exactly conserved for the
HBVM(15,3) method. Moreover, all methods exhibit the same order
(i.e., 6), with the error constant of the HBVM(4,3) and HBVM(15,3)
methods much smaller than that of the symplectic GAUSS6 method.
\begin{figure}
\centerline{\includegraphics[width=7cm,height=6cm]{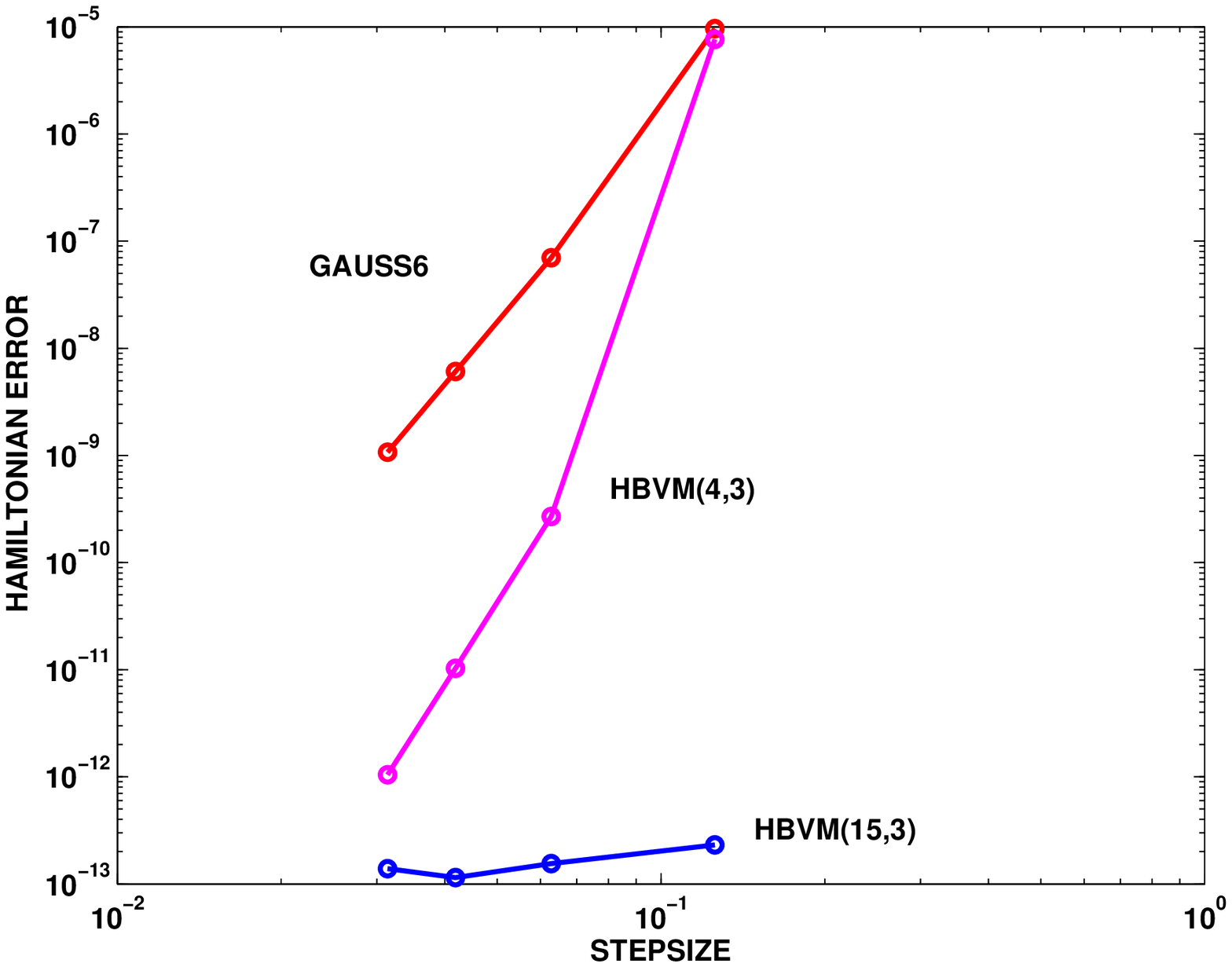}\qquad
\includegraphics[width=7cm,height=6cm]{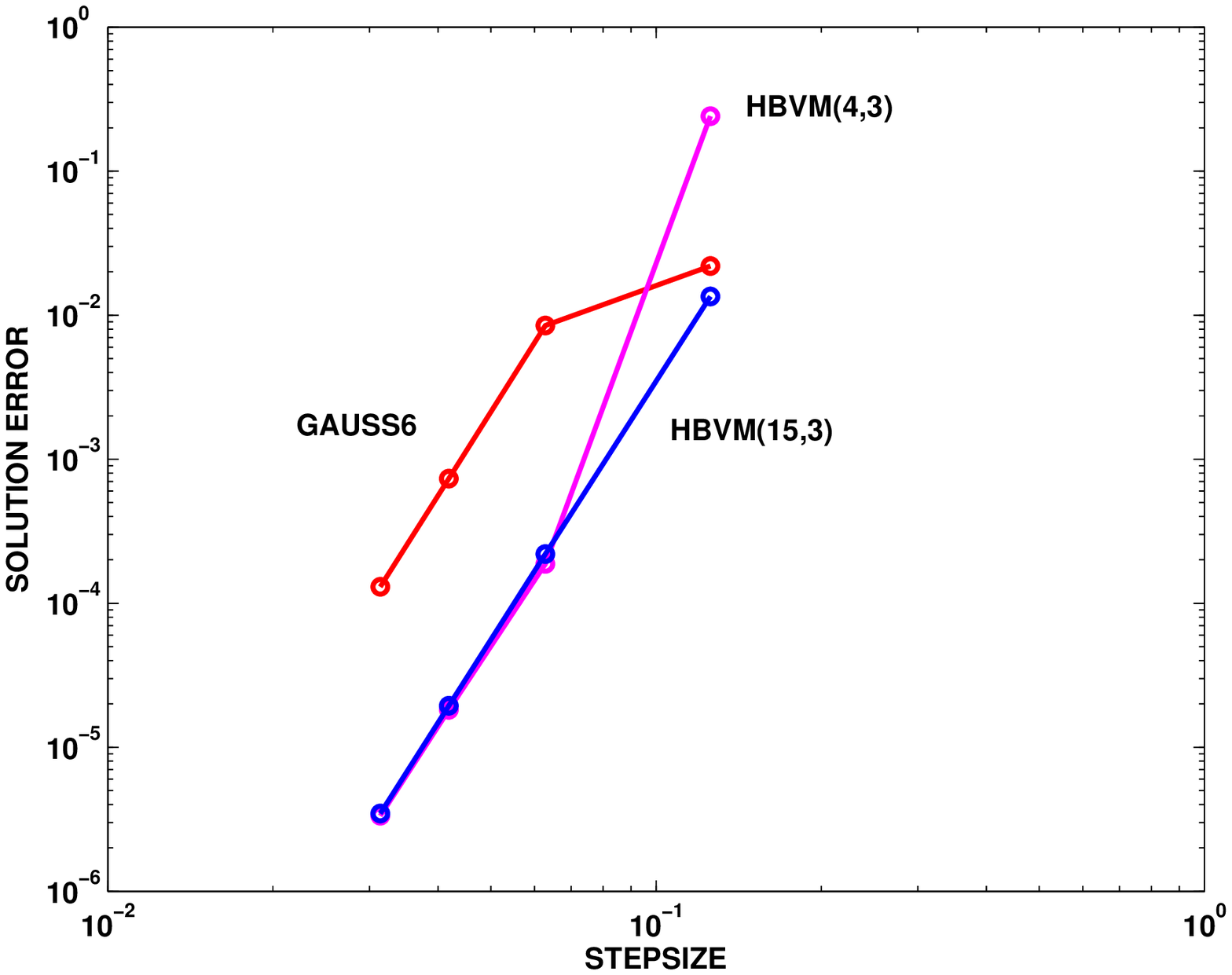}}
\caption{Kepler problem, $e=0.6$, Hamiltonian (left plot) and
solution (right plot) errors over 1000 periods with a constant
stepsize.} \label{cost_hy} \smallskip
\centerline{\includegraphics[width=7cm,height=6cm]{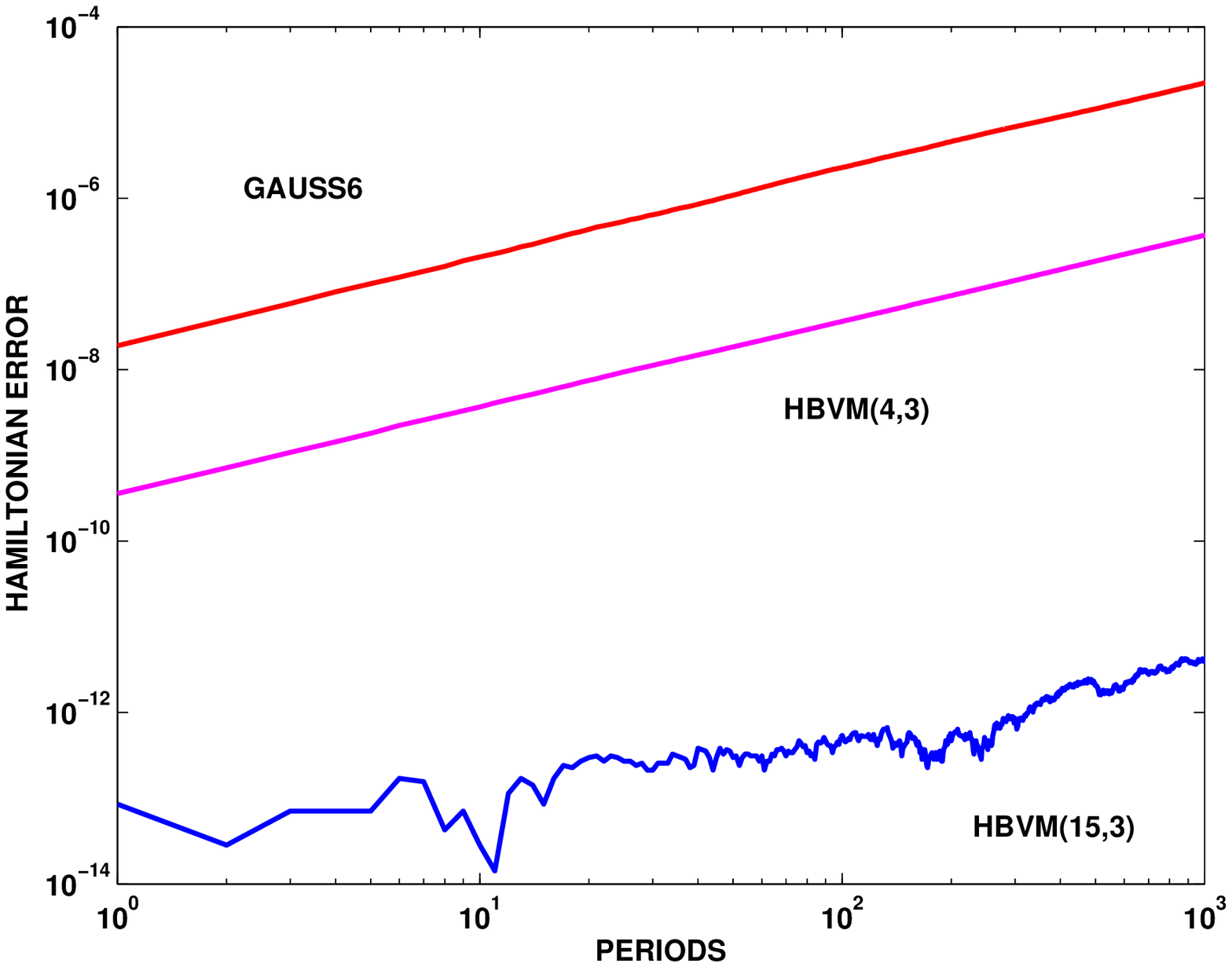}
\qquad
\includegraphics [width=7cm,height=6cm]{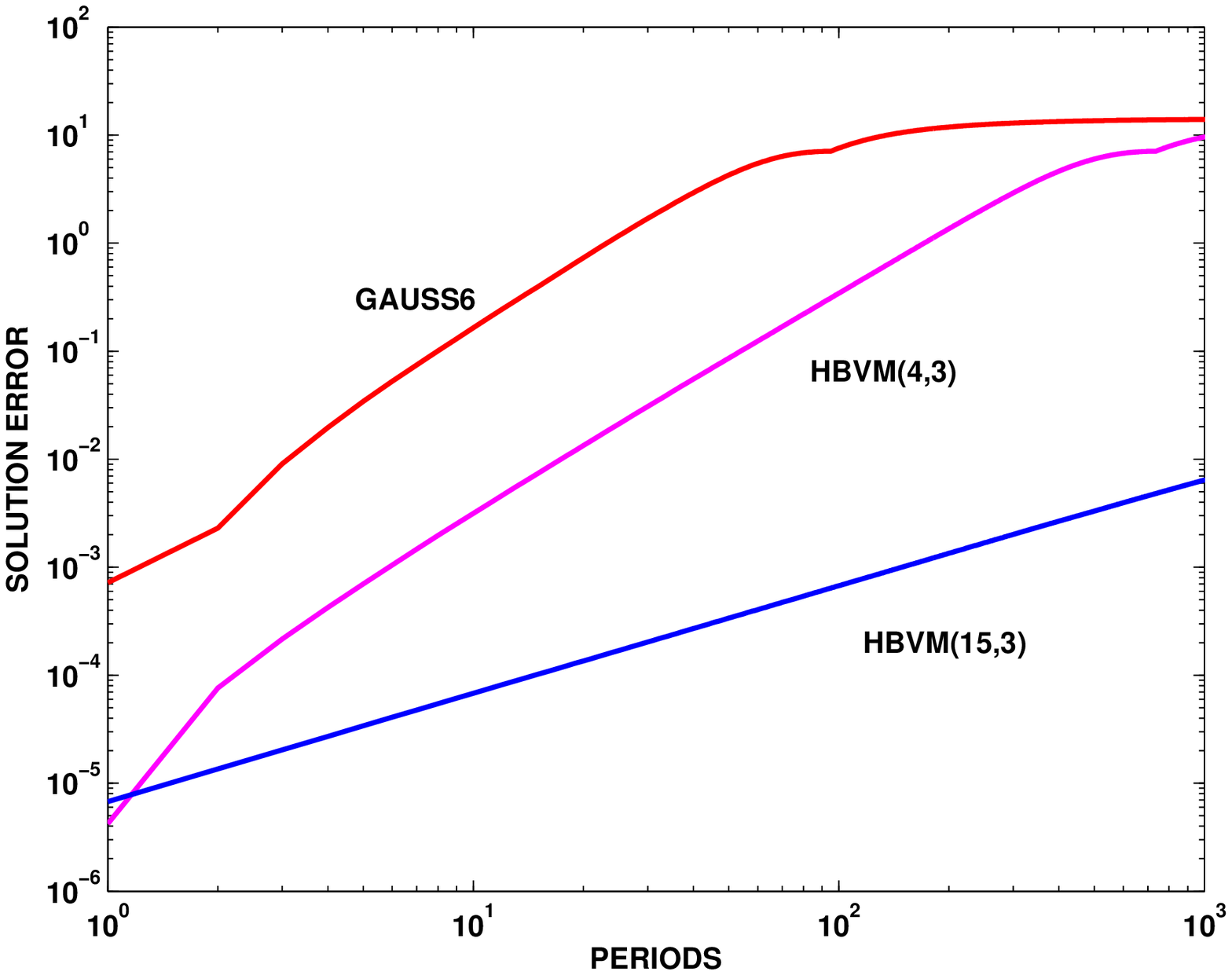}}
\caption{Kepler problem, $e=0.99$, Hamiltonian (left plot) and
solution (right plot) errors over 1000 periods with a variable
stepsize, $Tol = 10^{-10}$. Note that HBVM(4,3) is not energy
preserving, on the contrary of HBVM(15,3).} \label{var_hy}
\end{figure}
On the other hand, when $e=0.99$, we consider a variable stepsize
implementation with the following {\em standard} mesh-selection
strategy:
$$h_{new} = 0.7 \cdot h_n\left( \frac{Tol}{err_n}
\right)^{1/(p+1)},
$$
where $p=6$ is the order of the method, $Tol$
is the used tolerance, $h_n$ is the current stepsize, and $err_n$ is
an estimate of the local error. According to what stated in the
literature, see, e.g., \cite[pp.\,125--127]{SC},
\cite[p.\,235]{LeRe}, \cite[pp.\,303--305]{HLW}, this is not an
advisable choice for symplectic methods, for which a {\em drift} in
the Hamiltonian appears, and a {\em quadratic} error growth is
experienced, as is confirmed by the plots in Figure~\ref{var_hy}.
The same happens for the method HBVM(4,3), which is not energy
preserving. Conversely, for the (practically) energy preserving
method HBVM(15,3), {\em no drift} in the Hamiltonian occurs and a
{\em linear} error growth is observed.

\begin{remark} We observe that more refined (though more involved\,) mesh
selection strategies exist for symplectic methods  aimed to avoid
the numerical drift in the Hamiltonian and the quadratic error
growth (see, e.g., \cite[Chapter VIII]{HLW}). However, we want to
emphasize that they are no more necessary, when using energy
preserving methods, since obviously no drift can occur, in such a
case.\end{remark}

\section{Conclusions}\label{final}

In this paper, we have presented a general framework for the
derivation and analysis of effective one-step methods, which is
based on a local Fourier expansion of the problem at hand. In
particular, when the chosen basis is a polynomial one, we obtain a
large subclass of Runge-Kutta methods, which can be regarded as a
generalization of Gauss collocation methods.

When dealing with canonical Hamiltonian problems, the methods
coincide with the recently introduced class of energy preserving
methods named HBVMs. A few numerical tests seem to show that such
methods are less sensitive to a wider class of perturbations, with
respect to symplectic, or symmetric but non energy conserving,
methods. As matter of fact, on the contrary of the latter methods,
they can be conveniently associated with standard mesh selection
strategies.

\end{document}